\documentclass[12pt]{article}
\usepackage{amsfonts}
\usepackage{amssymb}
\usepackage{bbm}
\usepackage{amscd}
\usepackage{mathrsfs}
\usepackage{amsmath,amssymb, amsthm}
\usepackage{graphicx}
\usepackage{color}
\usepackage{authblk}
\usepackage{subcaption}
\usepackage{indentfirst,latexsym,amssymb}
\usepackage{tikz}
\usepackage{xcolor}
\usepackage[normalem]{ulem}
\usepackage[bookmarks=false]{hyperref}

\newtheorem{theorem}{Theorem}[section]
\newtheorem{pro}[theorem]{Proposition}

\newtheorem{lem}[theorem]{Lemma}

\newtheorem{coro}[theorem]{Corollary}

\theoremstyle{definition}
\newtheorem*{rem}{Remark}
\newtheorem{exam}[theorem]{Example}
\newtheorem{defi}[theorem]{Definition}
\def\DD{D}

\def\00{\mathbf{0}}

\def\Z{\mathbb{Z}}

\def\End{\hbox{\rm End}}

\def\DD{{\mathcal D}}

\begin{document}
\title{Flow modules and nowhere-zero flows}
\author{Junyang Zhang and Na Lu}
\affil{School of Mathematical Sciences, Chongqing Normal University, Chongqing 401331, P. R. China}
\date{}

\openup 0.5\jot
\maketitle
 \footnotetext{E-mail address: jyzhang@cqnu.edu.cn (Junyang Zhang), 528714499@qq.com (Na Lu)}
\begin{abstract}
Let $\Gamma$ be a graph, $A$ an abelian group, $\DD$ a given orientation of  $\Gamma$ and $R$ a unital subring of the endomorphism ring of $A$. It is shown that the set of all maps $\varphi$ from $E(\Gamma)$ to $A$ such that $(\DD,\varphi)$ is an $A$-flow forms a left $R$-module. Let $\Gamma$ be a union of two subgraphs $\Gamma_{1}$ and $\Gamma_{2}$, and $p^n$ a prime power. It is proved that $\Gamma$ admits a nowhere-zero $p^n$-flow if $\Gamma_{1}$ and $\Gamma_{2}$ have at most $p^n-2$ common edges and both have nowhere-zero $p^n$-flows. More important, it is proved that $\Gamma$ admits a nowhere-zero $4$-flow if $\Gamma_{1}$ and $\Gamma_{2}$ both have nowhere-zero $4$-flows and their common edges induce a connected subgraph of $\Gamma$ of size at most $3$. It is also proved that a graph of which every edge is contained in a cycle of length at most $4$ admits a nowhere-zero 4-flow.


\end{abstract}
\section{Introduction}
\label{sec:int}
Throughout this paper, for any abelian group, the operation is written additively and the identity element is denoted as $0$. All graphs considered are finite and undirected with no loops, possibly with multiple edges. For a graph $\Gamma$, we use $V(\Gamma)$ and $E(\Gamma)$ to denote its vertex set and edge set respectively. An \emph{orientation} $\DD$ of a graph $\Gamma$ is a digraph obtained from $\Gamma$ by endowing each edge of $\Gamma$ with one of the two possible orientations. For a vertex $v\in V(\Gamma)$, we use $\DD^{+}(v)$ to denote the set of edges of $\Gamma$ being endowed orientation with tail at $v$ and $\DD^{-}(v)$ the set of edges of $\Gamma$ being endowed orientation with head at $v$. Let $A$ be an abelian group, and $\varphi$ a map from $E(\Gamma)$ to $A$. Write
\begin{equation*}
  \varphi^{+}(v):=\sum\limits_{e\in \DD^{+}(v)}\varphi(e)~~\mbox{and}~~
  \varphi^{-}(v):=\sum\limits_{e\in \DD^{-}(v)}\varphi(e)
\end{equation*}
for every $v\in V(\Gamma)$, where the sums are meant to be group-theoretical addition performed in $A$.
The ordered pair $(\DD, \varphi)$ is called an $A$-\emph{flow} of $\Gamma$ if $\varphi^{+}(v)=\varphi^{-}(v)$ for all $v\in V(\Gamma)$. If in addition $\varphi(e)\neq0$ for every edge $e\in E(\Gamma)$, then $(\DD, \varphi)$ is called a \emph{nowhere-zero} $A$-flow. Use $\Z$ to denote the additive group of all integers and $\Z_{k}$ the additive group of integers module $k$ for every positive integer $k$. We call a $\Z$-flow an integer flow. An integer flow $(\DD, \varphi)$ of $\Gamma$ is called a $k$-flow if $|\varphi(e)|<k$ for all $e\in E(\Gamma)$. It is well known that a graph admits a nowhere-zero $k$-flow
if and only if it admits a nowhere-zero $\Z_{k}$-flow(see \cite[Theorem 1.3.3]{Z1997}), and if $(\DD, \varphi)$ is a nowhere-zero $k$-flow of $\Gamma$ then for any orientation $\DD'$ of $\Gamma$ there exists an integer value map $\varphi'$ on $E(\Gamma)$ such that $(\DD', \varphi')$ is a nowhere-zero $k$-flow of $\Gamma$ (see \cite[Theorem 1.2.8]{Z1997}). In this sense,  we call a nowhere-zero $\Z_{k}$-flow a nowhere-zero $k$-flow and we can choose any orientation when determining whether a graph admits a nowhere-zero $k$-flow.

In \cite{T1954}, Tutte initiated the study of nowhere-zero flows by proving that
a planar graph admits a nowhere-zero $k$-flow if and only if it is face $k$-colorable. Tutte made three beautiful conjectures, namely the $5$-flow, $4$-flow, and $3$-flow conjectures
(see \cite[Conjecture 1.1.5, 1.1.6 and 1.1.8]{Z1997}). These conjectures have motivated a great deal of research on this subject, but despite this all three remain unsolved. The best approach to the $5$-flow conjecture is due to Seymour \cite{S1981} who proved that every 2-edge-connected graph has a nowhere-zero $6$-flow.
In 1979, Jaeger \cite{J1979} proved that every $4$-edge-connected graph admits a nowhere-zero $4$-flow, and he further conjectured that there is a positive integer $k$ such that every $k$-edge-connected graph admits a nowhere-zero $3$-flow. Jaeger's conjecture was confirmed by Thomassen \cite{Th2012} who proved that the statement is true when $k=8$. This breakthrough was further improved by Lov\'asz, Thomassen, Wu and Zhang \cite{LTWZ2013} who proved that every $6$-edge-connected graph admits a nowhere-zero $3$-flow. For a comprehensive general treatment of the subject, we recommend the
monograph \cite{Z1997} by Zhang.

In \cite{IS2003}, Imrich and \v Skrekovski proved that the Cartesian product of any two nontrivial connected graphs has a nowhere-zero $4$-flow, and that it has a nowhere-zero $3$-flow provided that both factors are bipartite. Shu and Zhang \cite{SZ2005} later improved their result by showing that a nontrivial Cartesian product graph $\Gamma\square \Sigma$ has a nowhere-zero $3$-flow except when $\Gamma$ has a bridge and $\Sigma$ is an odd-circuit tree. Rollov\'a and \v Skoviera \cite{RS2012} generalised the result of Imrich and \v Skrekovski to Cartesian graph bundles (see Definition \ref{CB}) by proving that every Cartesian bundle of two graphs without isolated vertices has a nowhere-zero $4$-flow.

Recall that the \emph{Cartesian product} $\Sigma\square\Sigma'$ of two graphs $\Sigma$ and $\Sigma'$ is the graph defined by the rule:
\begin{itemize}
  \item its vertex-set is $V(\Sigma)\times V(\Sigma')$;
  \item its edge-set is $V(\Sigma)\times E(\Sigma')\cup E(\Sigma)\times V(\Sigma')$;
  \item for every edge $(u,e')\in V(\Sigma)\times E(\Sigma')$, the two ends of $(u,e')$ are $(u,u')$ and $(u,v')$ where $u'$ and $v'$ are the two ends of $e'$;
  \item for every edge $(e,u')\in E(\Sigma)\times V(\Sigma')$, the two ends of $(e,u')$ are $(u,u')$ and $(v,u')$ where $u$ and $v$ are the two ends of $e$.
\end{itemize}
It is obvious that every edge of a Cartesian product of two nontrivial connected graphs is contained in a cycle of length $4$. Inspired by the result of Imrich and \v Skrekovski \cite{IS2003}, we planed to verify Tutte's $4$-flow conjecture for all graphs of which every edge is contained in a cycle of length $4$. However, we found this problem being trivial if a lemma of Catlin \cite[Lemma 3.8.11]{Z1997} is applied. This Catlin's lemma is stated as follows: A graph admits a nowhere-zero $4$-flow if it is a union of a cycle of length at most $4$ and a subgraph admitting a nowhere-zero $4$-flow. Motivated by the Catlin's lemma, we study the integer flows of graphs which is a union of two subgraphs with a few number of common edges. From analysing the modular properties of all $A$-flows of a graph with given orientation, we study the nowhere-zero $A$-flows for the case that $A$ is an abelian group of prime power order. In particular, we study the nowhere-zero $4$-flows by giving a generalization of the Catlin's lemma.
The main results of this paper is summarized  as follows.

Firstly we analyse the structure of the set consisting of all $A$-flows of a graph with given orientation.
Let $\Gamma$ be a graph, $\DD$ a given orientation of  $\Gamma$, $A$ an abelian group, and $F$ a set consisting of all functions from $E(\Gamma)$ to $A$ such that $(\DD,\varphi)$ is an $A$-flow of $\Gamma$ for every $\varphi\in F$. Define a binary operation `$+$' by the rule:
$(\varphi_{1}+\varphi_{2})(a)=\varphi_{1}(a)+\varphi_{2}(a),~ \forall~\varphi_{1},\varphi_{2}\in F, ~a\in A$. It is obvious that $(F,+)$ is an additive group with zero function as its zero element. We will further show that $F$ forms a left $R$-module for any unital subring $R$ of the endomorphism ring $\End(A)$ of $A$ (Theorem \ref{module}).  In particular, if $A$ is the additive group of the finite field $\mathbb{F}_{p^{n}}$, then $F$ forms a linear space on the field $\mathbb{F}_{p^{n}}$.

Then, by using the theory of flow module, we study the nowhere-zero $p^n$-flow of graphs. Let $\Gamma$ be a union of two subgraphs $\Gamma_{1}$ and $\Gamma_{2}$, and $p^n$ a prime power. We prove that $\Gamma$ admits a nowhere-zero $p^n$-flow if $\Gamma_{1}$ and $\Gamma_{2}$ have at most $p^n-2$ common edges and both have nowhere-zero $p^n$-flows (see Theorem \ref{common} for a more general result).

Finally, we prove that $\Gamma:=\Gamma_{1}\cup\Gamma_{2}$ admits a nowhere-zero $4$-flow if $\Gamma_{1}$ and $\Gamma_{2}$ both have nowhere-zero $4$-flows and their common edges induce a connected subgraph of $\Gamma$ of size (number of edges) at most $3$ (Theorem \ref{union}). This result, together with Theorem \ref{common}, leads directly to the Catlin's lemma. By using Theorem \ref{union}, we prove that a graph of which every edge is contained in a cycle of length at most $4$ admits a nowhere-zero 4-flow (Theorem \ref{4f}). Note that our results can be used to prove the main result in \cite{RS2012} that every Cartesian bundle of two graphs without isolated vertices has a nowhere-zero $4$-flow (see the proof of Theorem \ref{bundle}).

\section{Preliminaries}
In this section, we fix some notations and introduce some concepts and lemmas for later use.

Let $S$ and $T$ be two sets. We use $S-T$ to denote the difference of $S$ and $T$, and $|S|$ to denote the cardinality of $S$. Let $\Gamma$ be a graph. For any subgraph $\Sigma$ of $\Gamma$, we use $O_{\Gamma}(\Sigma)$ to denote the set consisting of all vertices in $V(\Sigma)$ which have odd valency in $\Gamma$. Set $\overline{O}_{\Gamma}(\Sigma):=V(\Sigma)-O_{\Gamma}(\Sigma)$, and for simplicity write $O(\Gamma):=O_{\Gamma}(\Gamma)$ and $\overline{O}(\Gamma):=\overline{O}_{\Gamma}(\Gamma)$. If $\overline{O}(\Gamma)=V(\Gamma)$, then $\Gamma$ is called an \emph{even graph}.
If $\Sigma$ is a spanning subgraph of $\Gamma$ and $d_{\Sigma}(v)\equiv d_{\Gamma}(v)\pmod2$ for each vertex $v\in V(\Gamma)$, then $\Sigma$ is called a \emph{parity} subgraph of $\Gamma$. A collection of edge-disjoint subgraphs of $\Gamma$ is called a \emph{decomposition} of $\Gamma$ if every edge of $\Gamma$ belongs to exactly one member of this collection.
A decomposition of $\Gamma$ is called a
\emph{parity subgraph decomposition} if each of its members is a parity subgraph of $\Gamma$, and it is said to be \emph{nontrivial} if it has at least two members. Let $\Sigma_{1}$ and $\Sigma_{2}$ be two subgraphs of $\Gamma$. The \emph{symmetric difference} of $\Sigma_{1}$ and $\Sigma_{2}$, denoted by $\Sigma_{1}\triangle\Sigma_{2}$, is the graph with vertex set $V(\Sigma_{1})\cup V(\Sigma_{2})$ and
edge set $\big(E(\Sigma_{1})\cup E(\Sigma_{2})\big)-\big(E(\Sigma_{1})\cap E(\Sigma_{2})\big)$. It is straightforward to check that the symmetric difference of two even graphs is an even graph.
A path (cycle) is called a $k$-\emph{path}
($k$-\emph{cycle}) if it is of length $k$.  For graph-theoretic terminology and notation not mentioned here, please refer to \cite{BM2008}

The following lemma may be well known. For self-contained, we give its proof.
\begin{lem}
\label{2path}
Let $\Gamma$ be a graph and $\Gamma'$ a graph obtained from $\Gamma$ by removing a $2$-path $u_{1}u_{2}u_{3}$ and adding a new edge jointing $u_{1}$ and $u_{3}$. If $\Gamma'$ admits a nowhere-zero $k$-flow, then so does $\Gamma$.
\end{lem}
\begin{proof}
Use $e_i$ to denote the edge in the removing $2$-path with ends $u_{i}$ and $u_{i+1}$ where $i=1,2$, and $e_0$ to denote the new added edge.
Let $(\DD',\varphi')$ be a nowhere-zero $k$-flow of $\Gamma'$. Without loss of generality, assume that $e_0\in \DD'^{+}(u_{1})$. Define an orientation $\DD$ on $\Gamma$ by the rule: `$e_1\in \DD^{+}(u_{1})$, $e_2\in \DD^{+}(u_{2})$, and the restriction of $\DD$ to $\Gamma-\{e_1,e_2\}$ is coincident with the restriction of $\DD'$ to $\Gamma-e_0$'. Then it is straightforward to check that $(\DD,\varphi)$ is a nowhere-zero $k$-flow of $\Gamma$ where $\varphi$ is the function on $E(\Gamma)$ satisfying $\varphi(e_1)=\varphi(e_2)=\varphi'(e_0)$ and $\varphi(e)=\varphi'(e)$ for every $e\in E(\Gamma)-\{e_1,e_2\}$.
\end{proof}
The following three lemmas can be found in \cite{Z1997}.
\begin{lem}
{\rm\cite[Theorem 3.2.4]{Z1997}}
\label{parity}
A graph admits a nowhere-zero $4$-flow if and only if it has a nontrivial parity subgraph decomposition.
\end{lem}
\begin{rem}
\label{3parity}
It is easy to prove that a graph $\Gamma$ has a parity subgraph decomposition with three members if it has a nontrivial parity subgraph decomposition. Actually, if $\Gamma$ is even, then $\Gamma$ has a nontrivial  parity subgraph decomposition $\Gamma=\Gamma\cup\emptyset_{\Gamma}\cup\emptyset_{\Gamma}$ where $\emptyset_{\Gamma}$ is the empty graph on $V(\Gamma)$. Now we assume $\Gamma$ is not even and has a nontrivial parity subgraph decomposition $\Gamma=\Gamma_{1}\cup\Gamma_{2}\cup\cdots\cup\Gamma_{s}$. Then $s$ is an odd integer at least $3$. Set $\Gamma'_{3}=\Gamma_{3}\cup\cdots\cup\Gamma_{s}$. Then $\Gamma$ has a nontrivial parity subgraph decomposition $\Gamma=\Gamma_{1}\cup\Gamma_{2}\cup\Gamma'_{3}$.
\end{rem}
\begin{lem}
{\rm\cite[Theorem 3.3.3]{Z1997}}
\label{evenly}
A graph $\Gamma$ admits a nowhere-zero $4$-flow if and only if it has an even spanning subgraph $\Sigma$ such that $|O_{\Gamma}(\Theta)|$ is an even integer for each connected component $\Theta$ of $\Sigma$.
\end{lem}
\begin{lem}
{\rm\cite[Corollary 2.7.2]{Z1997}}
\label{equal}
Let $\Gamma$ be a graph. If $A$ and $B$ are two abelian groups of equal order, then $\Gamma$ has a nowhere-zero $A$-flow if and only if it has a nowhere-zero $B$-flow.
\end{lem}

\section{Flow Modules of Graphs}
Let $R$ be a ring with $1$ as its multiplicative identity. We call an abelian group $M$ a \emph{left $R$-module} if there is an operation:
$R\times M\longrightarrow M, ~(r,x)\mapsto rx,~\forall r\in R~\mbox{and}~x\in M$ satisfying the following axioms:
\begin{equation*}
r(x+y)=rx+ry,~(r+s)x=rx+sx,~(rs)x=r(sx)~\mbox{and}~1x=x
\end{equation*}
for all $r,s\in R$ and $x, y\in M$.
The operation of the ring $R$ on $M$ is called \emph{scalar multiplication}.
A right $R$-module $M$ is defined similarly,  except that the  scalar multiplication takes the form $M\times R\longrightarrow M,~(x,r)\mapsto xr$, and the above axioms are written with scalars $r$ and $s$ on the right of $x$ and $y$. If $R$ is commutative, then left $R$-modules are same as right $R$-modules and are simply called $R$-modules.
\begin{exam}
Let $M$ be an abelian group and $R$ be a unital subring of $\End(M)$. Then it is straightforward to check that $M$ is a left $R$-module with the scalar multiplication defined by $rx=r(x)$ for all $r\in R$ and $x\in M$.
\end{exam}
\begin{theorem}\label{module}
Let $\Gamma$ be a graph, $\DD$ an orientation of $\Gamma$, $A$ an abelian group, and $F$ a set consisting of all functions from $\DD$ to $A$ such that $(\DD,\varphi)$ is an $A$-flow of $\Gamma$ for every $\varphi\in F$. Then $F$ forms a left $R$-module for any unital subring $R$ of $\End(A)$.
\end{theorem}
\begin{proof}
Define a binary operation `$+$' by the rule:
$(\varphi_{1}+\varphi_{2})(a)=\varphi_{1}(a)+\varphi_{2}(a),~ \forall~\varphi_{1},\varphi_{2}\in F, ~a\in A$.
It is straightforward to check that $(F,+)$ is an additive group with zero function as its zero element. Consider the operation
$R\times F\longrightarrow F, ~(r,\varphi)\mapsto r\varphi,~\forall r\in R~\mbox{and}~\varphi\in F$  where $(r\varphi)(e)=r\varphi(e)$ for all $e\in E(\Gamma)$.
Take $r_{1},r_{2}\in R$ and $\varphi_{1},\varphi_{2}\in F$. Then
\begin{equation*}
\big(r_{1}(\varphi_{1}+\varphi_{2})\big)(e)=r_{1}(\varphi_{1}+\varphi_{2})(e)
=r_{1}\big (\varphi_{1}(e)+\varphi_{2}(e)\big)
=(r_{1} \varphi_{1}+r_{1}\varphi_{2})(e),
\end{equation*}
\begin{equation*}
((r_{1} +r_{2})\varphi_{1})(e)=(r_{1} +r_{2})\varphi_{1}(e)=r_{1} \varphi_{1}(e)+r_{2}\varphi_{1}(e)=(r_{1} \varphi_{1}+r_{2}\varphi_{1})(e),
\end{equation*}
\begin{equation*}
((r_{1}r_{2})\varphi_{1})(e)=(r_{1}r_{2})(\varphi_{1}(e))
=r_{1}\big(r_{2}\varphi_{1}(e)\big)=r_{1}(r_{2}\varphi_{1})(e)
=(r_{1}(r_{2}\varphi_{1}))(e)
\end{equation*}
and
\begin{equation*}
(1\varphi_{1})(e)=1(\varphi_{1}(e))=\varphi_{1}(e).
\end{equation*}
It follows that $r_{1}(\varphi_{1}+\varphi_{2})=r_{1}\varphi_{1}+r_{1} \varphi_{2}$, $(r_{1}+r_{2}) \varphi_{1}=r_{1} \varphi_{1}+r_{2} \varphi_{1}$,
$(r_{1}r_{2})\varphi_{1}=r_{1} (r_{2}\varphi_{1})$ and
$1 \varphi_{1}=\varphi_{1}$. Therefore $F$ is left $R$-module.
\end{proof}
\begin{coro}
Let $\Gamma$ be a graph, $A$ an abelian group and $m$ a positive integer. If $(\DD,\varphi_{i})$ is an $A$-flow of $\Gamma$ and $\sigma_{i}$ is an endomorphism of $A$ for every $1\leq i\leq m$, then $(\DD,\varphi)$ is an $A$-flow of $\Gamma$ where $\varphi=\sigma_{1}\varphi_{1}+\ldots+\sigma_{m}\varphi_{m}$.
\end{coro}
\begin{coro}\label{space}
Let $\Gamma$ be a graph, and $A$ the additive group of the finite field $\mathbb{F}_{p^{n}}$. Then all $A$-flows of $\Gamma$ with given orientation $\DD$ form a linear space on the field $\mathbb{F}_{p^{n}}$. In particular, the number of $A$-flows of $\Gamma$ with orientation $\DD$ is a power of $p^n$.
\end{coro}
\begin{proof}
Since $A$ is the additive group of the finite field $\mathbb{F}_{p^{n}}$, $\mathbb{F}_{p^{n}}$ can be seen as a unital subring of $\End(A)$ with the action of $\mathbb{F}_{p^{n}}$ on $A$ given by $(r,a)=ra$ for all $r\in\mathbb{F}_{p^{n}}$ and $a\in A$. Let $F$ be a set consisting of all functions from $E(\Gamma)$ to $A$ such that $(\DD,\varphi)$ is an $A$-flow of $\Gamma$ for every $\varphi\in F$. By Theorem $\ref{module}$,  $F$ is a $\mathbb{F}_{p^{n}}$-module. Since $\mathbb{F}_{p^{n}}$ is a field, $F$ is a linear space over $\mathbb{F}_{p^{n}}$. It follows that the cardinality of $F$ is a power of $p^n$.
\end{proof}

\begin{theorem}
\label{common}
Let $\Gamma$ be a graph, and $A$ an abelian group of order a prime power $p^{n}$. Suppose that there are $s$(~$\geq$2~) subgraphs  $\Gamma_{1},\Gamma_{2},...,\Gamma_{s}$ of $\Gamma$ such that $\Gamma=\cup_{i=1}^{s}\Gamma_{i}$. If $\Gamma_{i}$ admits a nowhere-zero $A$-flow for each $i\in\{1,2,...,s\}$, and $\cup_{i=1}^{\ell-1}\Gamma_{i}$ and $\Gamma_{\ell}$ have at most $p^{n}-2$ common edges for each $\ell\in\{2,...,s\}$, then $\Gamma$ admits a nowhere-zero $A$-flow.
\end{theorem}
\begin{proof}
By Lemma \ref{equal}, it is suffices to prove that the theorem is true for the case when $A$ is elementary abelian. Now we assume that $A$ is an elementary abelian group of order $p^n$ and identify it with the additive group of the finite field $\mathbb{F}_{p^{n}}$. We proceed the proof by induction on $s$.

The theorem is trivial for $s=1$. Now we assume $s\geq 2$ and the theorem is true for $s-1$. Set $\Theta_{1}=\cup_{i=1}^{s-1}\Gamma_{i}$ and $\Theta_{2}=\Gamma_{s}$. Then $\Theta_{2}$ admits a nowhere-zero $A$-flow, $\Gamma=\Theta_{1}\cup\Theta_{2}$, and $\Theta_{1}$ and $\Theta_{2}$ have at most $p^{n}-2$ common edges. By induction hypothesis, $\Theta_{1}$ admits a nowhere-zero $A$-flow. Let $\DD$ be an orientation of $\Gamma$, and $\DD_{i}$ the restriction of $\DD$ to $\Theta_{i}$ for $i=1,2$.
Since $\Theta_{i}$ admits a nowhere-zero $A$-flow, there exists a nowhere-zero $A$-flow $(\DD_{i},\varphi_{i})$ of  $\Theta_{i}$. Define functions $\varphi'_{i}:~E(\Gamma)\rightarrow A$ by the rule:
\begin{equation*}
 \varphi'_{i}(e)=\left\{
    \begin{array}{ll}
      \varphi_{i}(e), &\hbox{if}~e\in E(\Theta_{i}); \\
      0, &\hbox{if}~e\notin E(\Theta_{i}).
    \end{array}
  \right.
\end{equation*}
Then $(\DD,\varphi'_{i})$ is an $A$-flow of $\Gamma$.
By \cite[Theorem 4.6]{BJN1994}, the multiplicative group $\mathbb{F}_{p^{n}}^{*}$ of $\mathbb{F}_{p^{n}}$ is a cyclic group of order $p^{n}-1$. Let $b$ be a generator of $\mathbb{F}_{p^{n}}^{*}$. Let $e_{1},e_{2},\cdots,e_{k}$ be all the common edges of $\Theta_{1}$ and $\Theta_{2}$. Then $k\leq p^{n}-2$. Therefore there exists $j\in \{0,1,\ldots,p^{n}-2\}$ such that $\varphi'_{1}(e_{i})\neq b^j\varphi'_{2}(e_{i})$ for all $i\in\{1,\ldots,k\}$. Set $\varphi=\varphi'_{1}-b^j\varphi'_{2}$. By Corollary \ref{space}, $(\DD,\varphi)$ is an $A$-flow of $\Gamma$. Take an arbitrary $e\in E(\Gamma)$. If $e\in E(\Theta_{1})-E(\Theta_{2})$, then $\varphi(e)=\varphi_1(e)\neq 0$. If $e\in E(\Theta_{2})-E(\Theta_{1})$, then $\varphi(e)=\varphi_2(e)\neq 0$.
If $e\in E(\Theta_{1})\cap E(\Theta_{2})$, then $\varphi(e)=\varphi_1(e)-b^j\varphi_{2}(e)\neq 0$. Therefore $(\DD,\varphi)$ is a nowhere-zero $A$-flow of $\Gamma$.
\end{proof}
The following corollary includes some spacial cases of Theorem \ref{common}.
\begin{coro}
\label{kcommon}
Let $\Gamma$ be a graph, suppose that there are $n$ ($\geq 2$) subgraphs $\Gamma_{1},\Gamma_{2},...,\Gamma_{n}$ of $\Gamma$ such that $\Gamma=\cup_{i=1}^{n}\Gamma_{i}$, $\Gamma_{i}$ admits a nowhere-zero $k$-flow where $k=3,4~\mbox{or}~5$ for each $i\in\{1,2,...,n\}$, and $\cup_{i=1}^{l-1}\Gamma_{i}$ and $\Gamma_{l}$ have at most $k-2$ common edges for each $l\in\{2,...,n\}$, then $\Gamma$ admits a nowhere-zero $k$-flow.
\end{coro}
\section{Nowhere-zero $4$-flows}
A spacial case of Corollary \ref{kcommon} is the following proposition.
\begin{pro}
\label{4common}
Let $\Gamma_{1}$ and $\Gamma_{2}$ be two subgraphs of a graph $\Gamma$ such that $\Gamma=\Gamma_{1}\cup\Gamma_{2}$ and $|E(\Gamma_{1})\cap E(\Gamma_{2})|\leq2$. If both $\Gamma_{1}$ and $\Gamma_{2}$ have nowhere-zero $4$-flows, then $\Gamma$ admits a nowhere-zero $4$-flow.
\end{pro}
\begin{figure}[h]
  \centering
  \begin{tikzpicture}
\tikzstyle{every node}=[draw,shape=circle,label distance=-0.5mm,inner sep=1pt];
\node (0) at (-54:1.5) [fill,label=below right:\tiny{0}] {};
\node (1) at (18:1.5) [fill,label=below right:\tiny{1}] {};
\node (2) at (18+72:1.5) [fill,label=above:\tiny{2}] {};
\node (3) at (18+2*72:1.5) [fill,label=below left:\tiny{3}] {};
\node (4) at (18+3*72:1.5) [fill,label=below left:\tiny{4}] {};
\node (5) at (-54:0.8) [fill,label=right:\tiny{5}] {};
\node (6) at (18:0.8) [fill,label=below right:\tiny{6}] {};
\node (7) at (18+72:0.8) [fill,label=right:\tiny{7}] {};
\node (8) at (18+2*72:0.8) [fill,label=below left:\tiny{8}] {};
\node (9) at (18+3*72:0.8) [fill,label=left:\tiny{9}] {};
\draw
(3)--(4)--(0)(5)--(7)--(9)--(6)--(8)
(5)(1)--(6)(2)--(7)(4)--(9);
\draw[thick] (0)--(1)(8)--(5)(2)--(3);
\end{tikzpicture}
\begin{tikzpicture}
\tikzstyle{every node}=[draw,shape=circle,label distance=-0.5mm,inner sep=1pt];
\node (0) at (-54:1.5) [fill,label=below right:\tiny{0}] {};
\node (1) at (18:1.5) [fill,label=below right:\tiny{1}] {};
\node (2) at (18+72:1.5) [fill,label=above:\tiny{2}] {};
\node (3) at (18+2*72:1.5) [fill,label=below left:\tiny{3}] {};
\node (5) at (-54:0.8) [fill,label=right:\tiny{5}] {};
\node (8) at (18+2*72:0.8) [fill,label=below left:\tiny{8}] {};
\draw
(1)--(2)(0)--(5)(3)--(8);
\draw[thick] (0)--(1)(2)--(3)(8)--(5);
\end{tikzpicture}
\hspace{8mm}
\begin{tikzpicture}
\tikzstyle{every node}=[draw,shape=circle,label distance=-0.5mm,inner sep=1pt];
\node (0) at (-54:1.5) [fill,label=below right:\tiny{0}] {};
\node (1) at (18:1.5) [fill,label=below right:\tiny{1}] {};
\node (2) at (18+72:1.5) [fill,label=above:\tiny{2}] {};
\node (3) at (18+2*72:1.5) [fill,label=below left:\tiny{3}] {};
\node (4) at (18+3*72:1.5) [fill,label=below left:\tiny{4}] {};
\node (5) at (-54:0.8) [fill,label=right:\tiny{5}] {};
\node (6) at (18:0.8) [fill,label=below right:\tiny{6}] {};
\node (7) at (18+72:0.8) [fill,label=right:\tiny{7}] {};
\node (8) at (18+2*72:0.8) [fill,label=below left:\tiny{8}] {};
\node (9) at (18+3*72:0.8) [fill,label=left:\tiny{9}] {};
\draw
(1)--(2)--(3)--(4)--(0)(7)--(9)--(6)--(8)--(5)
(1)--(6)(3)--(8)(4)--(9);
\draw[thick] (0)--(1)--(2)(5)--(7);
\end{tikzpicture}
\begin{tikzpicture}
\tikzstyle{every node}=[draw,shape=circle,label distance=-0.5mm,inner sep=1pt];
\node (0) at (-54:1.5) [fill,label=below right:\tiny{0}] {};
\node (1) at (18:1.5) [fill,label=below right:\tiny{1}] {};
\node (2) at (18+72:1.5) [fill,label=above:\tiny{2}] {};
\node (5) at (-54:0.8) [fill,label=right:\tiny{5}] {};
\node (7) at (18+72:0.8) [fill,label=right:\tiny{7}] {};
\draw
(0)--(5)(2)--(7);
\draw[thick] (0)--(1)--(2)(5)--(7);
\end{tikzpicture}
\caption{Subgraphs of Peterson graph}\label{peterson}
\end{figure}
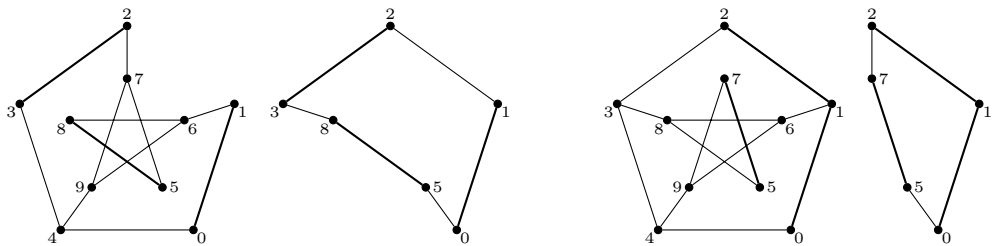
Note that the upper bound of $|E(\Gamma_{1})\cap E(\Gamma_{2})|$ in
Proposition \ref{4common} is sharp. For example, in Figure \ref{peterson}, the Peterson graph is both the union of the first two graphs and the union of the last two graphs. The set of the common edges of the first two graphs forms a matching with $3$-edges and that of the last two graphs is a disjoint union of an edge and a $2$-path. It is straightforward to check that all the four subgraphs in Figure \ref{peterson} admit nowhere-zero $4$-flows. However, the Peterson graph does not admit any nowhere-zero $4$-flow.

The following theorem shows that the upper bound of $|E(\Gamma_{1})\cap E(\Gamma_{2})|$ in Proposition \ref{4common} can be improved to $3$ under the condition that the subgraph induced by $E(\Gamma_{1})\cap E(\Gamma_{2})$ is connected.
\begin{theorem}
\label{union}
Let $\Gamma_{1}$ and $\Gamma_{2}$ be two subgraphs of a graph $\Gamma$ such that $\Gamma=\Gamma_{1}\cup\Gamma_{2}$, $|E(\Gamma_{1})\cap E(\Gamma_{2})|\leq3$ and $E(\Gamma_{1})\cap E(\Gamma_{2})$ induces a connected subgraph of $\Gamma$. If both $\Gamma_{1}$ and $\Gamma_{2}$ have nowhere-zero $4$-flows, then $\Gamma$ admits a nowhere-zero $4$-flow.
\end{theorem}
\begin{proof}
We proceed the proof by induction on the size of $\Gamma$.
The theorem follows Proposition \ref{4common} if $|E(\Gamma_{1})\cap E(\Gamma_{2})|\leq2$. Now we assume $|E(\Gamma_{1})\cap E(\Gamma_{2})|=3$.
We may also assume that $\Gamma_{1}$ and $\Gamma_{2}$ are both connected (for otherwise we consider their connected components). Since both $\Gamma_{1}$ and $\Gamma_{2}$ have nowhere-zero $4$-flows and $\Gamma=\Gamma_{1}\cup\Gamma_{2}$, $\Gamma$ is $2$-edge-connected. By \cite[Proposition 3.1]{L1995}, a $2$-edge-connected graph of order less than $17$ either admits a nowhere-zero $4$-flow or can be contracted to the Petersen graph. Therefore the theorem holds for graphs of size less than $15$. Suppose that $\Gamma$ is of size at least $15$ and the theorem holds for graphs of size less than that of $\Gamma$. We will prove the theorem by confirming that $\Gamma$ admits a nowhere-zero $4$-flow.

Since $\Gamma_{i}$ admits a nowhere-zero $4$-flow, by Lemma \ref{parity} and the remark below it, $\Gamma_{i}$ has a parity subgraph decomposition
$\Gamma_{i}=\Gamma_{i1}\cup\Gamma_{i2}\cup\Gamma_{i3}$ where $i=1,2$. Define a bipartite graph $\Sigma$ with two parts $\{11,12,13\}$ and $\{21,22,23\}$ by the rule:`$1i$ and $2j$ are linked by $k$ edges if and only if $\Gamma_{1i}$ and $\Gamma_{2j}$ have $k$ common edges, $i,j=1,2,3$'. Since $|E(\Gamma_{1})\cap E(\Gamma_{2})|=3$, the size of $\Sigma$ is $3$. Therefore, up to isomorphism, $\Sigma$ is one of the seven graphs in Figure \ref{sigma}.
\begin{figure}[h]
\centering
\begin{tikzpicture}
\tikzstyle{every node}=[draw,shape=circle,label distance=-0.5mm,inner sep=1pt];
\node (11) at (0,1) [fill,label=above:\tiny{11}] {};
\node (12) at (0,0) [fill,label=above:\tiny{12}] {};
\node (13) at (0,-1) [fill,label=below:\tiny{13}] {};
\node (21) at (1,1) [fill,label=above:\tiny{21}] {};
\node (22) at (1,0) [fill,label=above:\tiny{22}] {};
\node (23) at (1,-1) [fill,label=below:\tiny{23}] {};
\draw
(11) -- (21)(11)..controls (0.5,1.25)..(21)(11)..controls (0.5,0.75)..(21);
\end{tikzpicture}
\hspace{6mm}
\begin{tikzpicture}
\tikzstyle{every node}=[draw,shape=circle,label distance=-0.5mm,inner sep=1pt];
\node (11) at (0,1) [fill,label=above:\tiny{11}] {};
\node (12) at (0,0) [fill,label=above:\tiny{12}] {};
\node (13) at (0,-1) [fill,label=below:\tiny{13}] {};
\node (21) at (1,1) [fill,label=above:\tiny{21}] {};
\node (22) at (1,0) [fill,label=above:\tiny{22}] {};
\node (23) at (1,-1) [fill,label=below:\tiny{23}] {};
\draw
(12) -- (22)(11)..controls (0.5,1.25)..(21)(11)..controls (0.5,0.75)..(21);
\end{tikzpicture}
\hspace{6mm}
\begin{tikzpicture}
\tikzstyle{every node}=[draw,shape=circle,label distance=-0.5mm,inner sep=1pt];
\node (11) at (0,1) [fill,label=above:\tiny{11}] {};
\node (12) at (0,0) [fill,label=above:\tiny{12}] {};
\node (13) at (0,-1) [fill,label=below:\tiny{13}] {};
\node (21) at (1,1) [fill,label=above:\tiny{21}] {};
\node (22) at (1,0) [fill,label=above:\tiny{22}] {};
\node (23) at (1,-1) [fill,label=below:\tiny{23}] {};
\draw
(11) -- (22)(11)..controls (0.5,1.25)..(21)(11)..controls (0.5,0.75)..(21);
\end{tikzpicture}
\hspace{6mm}
\begin{tikzpicture}
\tikzstyle{every node}=[draw,shape=circle,label distance=-0.5mm,inner sep=1pt];
\node (11) at (0,1) [fill,label=above:\tiny{11}] {};
\node (12) at (0,0) [fill,label=above:\tiny{12}] {};
\node (13) at (0,-1) [fill,label=below:\tiny{13}] {};
\node (21) at (1,1) [fill,label=above:\tiny{21}] {};
\node (22) at (1,0) [fill,label=above:\tiny{22}] {};
\node (23) at (1,-1) [fill,label=below:\tiny{23}] {};
\draw
(11) -- (21)(11)--(22)(12)--(21);
\end{tikzpicture}
\hspace{6mm}
\begin{tikzpicture}
\tikzstyle{every node}=[draw,shape=circle,label distance=-0.5mm,inner sep=1pt];
\node (11) at (0,1) [fill,label=above:\tiny{11}] {};
\node (12) at (0,0) [fill,label=above:\tiny{12}] {};
\node (13) at (0,-1) [fill,label=below:\tiny{13}] {};
\node (21) at (1,1) [fill,label=above:\tiny{21}] {};
\node (22) at (1,0) [fill,label=above:\tiny{22}] {};
\node (23) at (1,-1) [fill,label=below:\tiny{23}] {};
\draw
(11) -- (21)(12)--(22)(13)--(23);
\end{tikzpicture}
\hspace{6mm}
\begin{tikzpicture}
\tikzstyle{every node}=[draw,shape=circle,label distance=-0.5mm,inner sep=1pt];
\node (11) at (0,1) [fill,label=above:\tiny{11}] {};
\node (12) at (0,0) [fill,label=above:\tiny{12}] {};
\node (13) at (0,-1) [fill,label=below:\tiny{13}] {};
\node (21) at (1,1) [fill,label=above:\tiny{21}] {};
\node (22) at (1,0) [fill,label=above:\tiny{22}] {};
\node (23) at (1,-1) [fill,label=below:\tiny{23}] {};
\draw
(11) -- (21)(11)--(22)(13)--(23);
\end{tikzpicture}
\hspace{6mm}
\begin{tikzpicture}
\tikzstyle{every node}=[draw,shape=circle,label distance=-0.5mm,inner sep=1pt];
\node (11) at (0,1) [fill,label=above:\tiny{11}] {};
\node (12) at (0,0) [fill,label=above:\tiny{12}] {};
\node (13) at (0,-1) [fill,label=below:\tiny{13}] {};
\node (21) at (1,1) [fill,label=above:\tiny{21}] {};
\node (22) at (1,0) [fill,label=above:\tiny{22}] {};
\node (23) at (1,-1) [fill,label=below:\tiny{23}] {};
\draw
(11) -- (21)(11)--(22)(11)--(23);
\end{tikzpicture}
\caption{$\Sigma$}\label{sigma}
\end{figure}

Firstly we assume that $\Sigma$ is not the last graph in Figure \ref{sigma}. Set
\begin{equation*}
\Gamma'_1=(\Gamma_{11}\cup\Gamma_{12})\triangle(\Gamma_{21}\cup\Gamma_{23})~
\mbox{and}~\Gamma'_2=(\Gamma_{11}\cup\Gamma_{13})\triangle
(\Gamma_{22}\cup\Gamma_{23}).
\end{equation*}
Then $\Gamma=\Gamma'_1\cup\Gamma'_2$. Since $\Gamma_{i1}$, $\Gamma_{i2}$ and $\Gamma_{i3}$ are pairwise edge-disjoint   parity subgraphs of $\Gamma_{i}$ for $i=1,2$, all the four graphs $\Gamma_{11}\cup\Gamma_{12}$, $\Gamma_{21}\cup\Gamma_{23}$,
$\Gamma_{11}\cup\Gamma_{13}$ and $\Gamma_{22}\cup\Gamma_{23}$ are even. It follows that both $\Gamma'_1$ and $\Gamma'_2$ are even graphs. It is well known \cite[Theorem 3.1.2]{Z1997} that a graph admits a nowhere-zero $4$-flow if and only if it is a union of two even graphs. Therefore $\Gamma$ admits a nowhere-zero $4$-flow.

Now we assume that $\Sigma$ is the last graph in Figure \ref{sigma}. Set
\begin{equation*}
\Gamma_{11}\cap\Gamma_{21}=\{e_1\},~ \Gamma_{11}\cap\Gamma_{22}=\{e_2\}~\mbox{and}~ \Gamma_{11}\cap\Gamma_{23}=\{e_3\}.
\end{equation*}
We divide the remainder proof into the following two cases.

\textsf{Case 1.} $\Gamma_{i_{0}j_{0}}$ contains a $2$-path $\Pi$ for some $i_{0}\in\{1,2\}$ and $j_{0}\in\{1,2,3\}$ such that $\{e_1,e_2,e_3\}\cap E(\Pi)=\emptyset$.

Let $u$ and $v$ be the two ends of the $2$-path $\Pi$. Let $\Gamma_{i_{0}j_{0}}'$ be a graph obtained from $\Gamma_{i_{0}j_{0}}$ by removing $\Pi$ and adding a new edge linking $u$ and $v$ (an edge not contained in $E(\Gamma)$). Write $\Gamma'_{ij}:=\Gamma_{ij}$ for $i\in\{1,2\}-\{i_{0}\}$ and
$j\in\{1,2,3\}-\{j_{0}\}$. Set $\Gamma'_{1}=\Gamma'_{11}\cup\Gamma'_{12}\cup\Gamma'_{13}$, $\Gamma'_{2}=\Gamma'_{21}\cup\Gamma'_{22}\cup\Gamma'_{23}$ and $\Gamma'=\Gamma'_{1}\cup\Gamma'_{2}$. Then $\Gamma'_{i1}$, $\Gamma'_{i2}$ and $\Gamma'_{i3}$ are pairwise edge-disjoint parity subgraphs of $\Gamma'_{i}$ for $i=1,2$. By Lemma \ref{parity}, both $\Gamma'_{1}$ and $\Gamma'_{2}$ have nowhere-zero $4$-flows. Note that $|E(\Gamma'_{1})\cap E(\Gamma'_{2})|=3$ and $\Gamma'$ is of size less than that of $\Gamma$. By induction hypothesis, $\Gamma'$ admits a nowhere-zero $4$-flow. Then, by Lemma \ref{2path},
$\Gamma$ admits a nowhere-zero $4$-flow.

\textsf{Case 2.} For any $i\in\{1,2\}$ and $j\in\{1,2,3\}$, if $\Gamma_{ij}$ contains a $2$-path $\Pi$, then  $\{e_1,e_2,e_3\}\cap E(\Pi)\neq\emptyset$.

\begin{figure}[h]
  \centering
\begin{tikzpicture}
\tikzstyle{every node}=[draw,shape=circle,label distance=-0.5mm,inner sep=1pt];
\node (11) at (0,0) [fill,label=below:\tiny{$u_1$}] {};
\node (12) at (1.2,0) [fill,label=below:\tiny{$u_2$}] {};
\draw
(11) -- (12)(11)..controls (0.6,0.4)..(12)(11)..controls (0.6,-0.4)..(12);
\end{tikzpicture}
\hspace{3mm}
\begin{tikzpicture}
\tikzstyle{every node}=[draw,shape=circle,label distance=-0.5mm,inner sep=1pt];
\node (11) at (0,0) [fill,label=below:\tiny{$u_1$}] {};
\node (12) at (1.2,0) [fill,label=below:\tiny{$u_2$}] {};
\node (13) at (2.4,0) [fill,label=below:\tiny{$u_3$}] {};
\draw
(11)--(12)(12)..controls (1.8,0.4)..(13)(12)..controls (1.8,-0.4)..(13);
\end{tikzpicture}
\hspace{3mm}
\begin{tikzpicture}
\tikzstyle{every node}=[draw,shape=circle,label distance=-0.5mm,inner sep=1pt];
\node (11) at (0,0) [fill,label=left:\tiny{$u_1$}] {};
\node (12) at (1.2,0.5) [fill,label=right:\tiny{$u_2$}] {};
\node (13) at (1.2,0) [fill,label=right:\tiny{$u_3$}] {};
\node (21) at (1.2,-0.5) [fill,label=right:\tiny{$u_4$}] {};
\draw
(11) -- (12)(11) --(13)(11) --(21);
\end{tikzpicture}
\hspace{3mm}
\begin{tikzpicture}
\tikzstyle{every node}=[draw,shape=circle,label distance=-0.5mm,inner sep=1pt];
\node (11) at (330:0.7) [fill,label=below right:\tiny{$u_1$}] {};
\node (12) at (90:0.7) [fill,label=above:\tiny{$u_2$}] {};
\node (13) at (210:0.7) [fill,label=below left:\tiny{$u_3$}] {};
\draw
(11)--(12)--(13)--(11);
\end{tikzpicture}
\hspace{3mm}
\begin{tikzpicture}
\tikzstyle{every node}=[draw,shape=circle,label distance=-0.5mm,inner sep=1pt];
\node (11) at (0,0) [fill,label=below:\tiny{$u_1$}] {};
\node (12) at (1.2,0) [fill,label=below:\tiny{$u_2$}] {};
\node (13) at (2.4,0) [fill,label=below:\tiny{$u_3$}] {};
\node (21) at (3.6,0) [fill,label=below:\tiny{$u_4$}] {};
\draw
(11) -- (12)--(13)--(21);
\end{tikzpicture}
  \caption{$\Gamma[e_1,e_2,e_3]$}\label{3edges}
\end{figure}
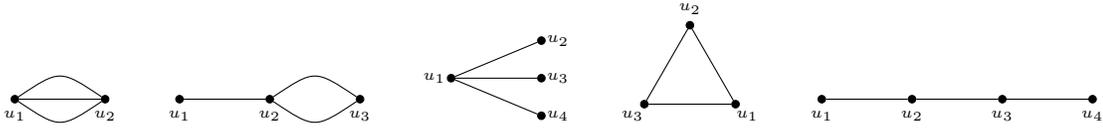
In this case, the following claim holds.
\textbf{Claim A}: for any $i\in\{1,2\}$ and $j\in\{1,2,3\}$, every connected component of $\Gamma_{ij}$ not containing edges in $\{e_1,e_2,e_3\}$ is either an isolated vertex or an edge.

Use $\Gamma[e_1,e_2,e_3]$ to denote the subgraph of $\Gamma$ induced by the edge set $\{e_1,e_2,e_3\}$. Since $\Gamma[e_1,e_2,e_3]$ is connected, it is one of the five graphs in Figure \ref{3edges}.

If $\Gamma[e_1,e_2,e_3]$ is not a simple graph, then it is the first or the second graph in Figure \ref{3edges}. Without loss of generality, we assume that $e_2$ and $e_3$ are parallel edges. Set $\Gamma'_{21}=\Gamma_{21}\cup \{e_2\}$, $\Gamma'_{22}=\Gamma_{22}-e_2$, $\Gamma'_{23}=\Gamma_{23}-e_3$ and $\Gamma'_{2}=\Gamma_{2}-e_3$. Then $\Gamma'_{2}$ has a parity subgraph decomposition $\Gamma'_{2}=\Gamma'_{21}\cup\Gamma'_{22}\cup\Gamma'_{23}$.
By Lemma \ref{parity}, $\Gamma'_{2}$ admits a nowhere-zero $4$-flow. Note that  $\Gamma=\Gamma_{1}\cup\Gamma'_{2}$ and $|E(\Gamma_{1})\cap E(\Gamma'_{2})|=2$.
By Proposition \ref{4common}, $\Gamma$ admits a nowhere-zero $4$-flow.

If $\Gamma[e_1,e_2,e_3]$ is the third graph in Figure \ref{3edges}, then it follows \textbf{Claim A} that $\Gamma_{11}\cup\Gamma_{12}$ has a connected component being a union of two cycles $\Lambda_{1}$ and $\Lambda_{2}$ with $u_1$ as their unique common vertex. Set $\Gamma'_{11}=\Gamma_{11}\triangle\Lambda_{1}$ and $\Gamma'_{12}=\Gamma_{12}\triangle\Lambda_{1}$. Then we obtain a new parity subgraph decomposition $\Gamma_{1}=\Gamma'_{11}\cup\Gamma'_{12}\cup\Gamma_{13}$ of $\Gamma_{1}$. It is straightforward to check that the graph $\Sigma$ being associated with the decompositions
$\Gamma_{1}=\Gamma'_{11}\cup\Gamma'_{12}\cup\Gamma_{13}$ and $\Gamma_{2}=\Gamma_{21}\cup\Gamma_{22}\cup\Gamma_{23}$ is isomorphic to the sixth graph in Figure \ref{sigma}. By what we have proved in paragraph 3, $\Gamma$ admits a nowhere-zero $4$-flow.

Let $\{i,j,k\}=\{1,2,3\}$. If the connected component $\Theta$ of $\Gamma_{2i}\cup\Gamma_{2j}$ containing $e_i$ does not contain $e_j$, then we obtain a new parity subgraph decomposition $\Gamma_{2}=\Gamma'_{2i}\cup\Gamma'_{2j}\cup\Gamma_{2k}$ of $\Gamma_{2}$ where $\Gamma'_{2i}:=\Gamma_{2i}\triangle\Theta$ and $\Gamma'_{2j}:=\Gamma_{2j}\triangle\Theta$. It is straightforward to check that the graph $\Sigma$ being associated with the decompositions
$\Gamma_{1}=\Gamma_{11}\cup\Gamma_{12}\cup\Gamma_{13}$ and $\Gamma_{2}=\Gamma'_{2i}\cup\Gamma'_{2j}\cup\Gamma_{2k}$ is isomorphic to the third graph in Figure \ref{sigma}, which boils down to the case we have dealt with.

From now on, we assume that $\Gamma[e_1,e_2,e_3]$ is the fourth or fifth graph in Figure \ref{3edges} and $\Gamma_{2i}\cup\Gamma_{2j}$ has a connected component containing both $e_i$ and $e_j$ where $1\leq i<j\leq3$. Since $\Gamma_{12}$ and $\Gamma_{13}$ are edge-disjoint parity subgraphs of $\Gamma_1$, $\Gamma_{12}\cup\Gamma_{13}$ is a spanning even subgraph of $\Gamma_1$. Similarly, $\Gamma_{2i}\cup\Gamma_{2j}$ is a spanning even subgraph of $\Gamma_2$. Since $\Gamma=\Gamma_1\cup\Gamma_2$ and $\Gamma_{12}\cup\Gamma_{13}$ and $\Gamma_{2i}\cup\Gamma_{2j}$ have no common edges, $(\Gamma_{12}\cup\Gamma_{13})\cup(\Gamma_{2i}\cup\Gamma_{2j})$ is a spanning even subgraph of $\Gamma$.
To complete the proof, we will firstly prove that there exists a pair of distinct integers $i,j\in\{1,2,3\}$ such that every connected component of
$(\Gamma_{12}\cup\Gamma_{13})\cup(\Gamma_{2i}\cup\Gamma_{2j})$ have even number
of odd vertices of $\Gamma$.
Let $\Theta_{ij}$ be an arbitrary connected component of $(\Gamma_{12}\cup\Gamma_{13})\cup(\Gamma_{2i}\cup\Gamma_{2j})$.
Set
\begin{equation*}
T_{ij}=V(\Theta_{ij})\cap O(\Gamma[e_1,e_2,e_3])~\mbox{and}~\overline{T}_{ij}=V(\Theta_{ij})\cap \overline{O}(\Gamma[e_1,e_2,e_3]).
\end{equation*}
Then
\begin{equation}\label{val}
d_{\Gamma}(u)=
\left\{\begin{array}{ll}
  d_{\Gamma_1}(u)+d_{\Gamma_2}(u), & \hbox{if}~u\in V(\Theta_{ij})-(T_{ij}\cup
  \overline{T}_{ij}); \\
  d_{\Gamma_1}(u)+d_{\Gamma_2}(u)-1, & \hbox{if}~u\in T_{ij}; \\
  d_{\Gamma_1}(u)+d_{\Gamma_2}(u)-2, & \hbox{if}~u\in \overline{T}_{ij}.
   \end{array}
\right.
\end{equation}
Set $\Theta'_{ij}=\Theta_{ij}\cap(\Gamma_{12}\cup\Gamma_{13})$ and $\Theta''_{ij}=\Theta_{ij}\cap(\Gamma_{2i}\cup\Gamma_{2j})$. Then $\Theta'_{ij}$ is a union of some connected components of $\Gamma_{12}\cup\Gamma_{13}$ and $\Theta''_{ij}$ is a union of some connected components of $\Gamma_{2i}\cup\Gamma_{2j}$. By \cite[Lemma 3.3.2]{Z1997}, both $|O_{\Gamma_1}(\Theta'_{ij})|$ and $|O_{\Gamma_2}(\Theta''_{ij})|$ are even.
Set
\begin{equation*}
R_{ij}=O_{\Gamma_1}(\Theta'_{ij})\cup O_{\Gamma_2}(\Theta''_{ij})~\mbox{and}~
S_{ij}=O_{\Gamma_1}(\Theta'_{ij})\cap O_{\Gamma_2}(\Theta''_{ij}).
\end{equation*}
Then $|R_{ij}-S_{ij}|=|O_{\Gamma_1}(\Theta'_{ij})|+
|O_{\Gamma_2}(\Theta''_{ij})|-2|S_{ij}|$.
Therefore $|R_{ij}-S_{ij}|$ is an even integer.
By (\ref{val}) we have
\begin{equation}\label{oddc}
\overline{O}_{\Gamma}(T_{ij})\subseteq R_{ij}-S_{ij}~\mbox{and}~O_{\Gamma}(T_{ij})\subseteq S_{ij}\cup(V(\Theta_{ij})-R_{ij}),
\end{equation}
and if $u\in V(\Theta_{ij})-T_{ij}$, then
\begin{equation}\label{evenc}
  u\in O_{\Gamma}(\Theta_{ij})\Leftrightarrow u\in R_{ij}-S_{ij}.
\end{equation}
If $\Gamma[e_1,e_2,e_3]$ is the fourth graph in Figure \ref{3edges}, then $T_{ij}=\emptyset$ for all $i$ and $j$ with $1\leq i<j\leq3$.
By (\ref{evenc}), $O_{\Gamma}(\Theta_{ij})=R_{ij}-S_{ij}$ and therefore  $|O_{\Gamma}(\Theta_{ij})|$ is even.
Now we assume that $\Gamma[e_1,e_2,e_3]$ is the fifth graph in Figure \ref{3edges}. Without loss of generality, we assume that the ends of $e_1$ are $u_1,u_2$ and the ends of $e_3$ are $u_3,u_4$. Since $\Gamma_{21}\cup\Gamma_{23}$ has a connected component containing both $e_1$ and $e_3$, either $T_{13}=\emptyset$ or $T_{13}=\{u_1,u_4\}$. It follows that
\begin{equation*}
(|O_{\Gamma}(T_{13})|,|\overline{O}_{\Gamma}(T_{13})|)
=(0,0),(1,1),(2,0)~\mbox{or}~(0,2).
\end{equation*}
By (\ref{oddc}) and (\ref{evenc}), we have
\begin{equation*}
O_{\Gamma}(\Theta_{13})=[(R_{13}-S_{13})-\overline{O}_{\Gamma}(T_{13})]
\cup O_{\Gamma}(T_{13})
\end{equation*}
and therefore
\begin{equation*}
|O_{\Gamma}(\Theta_{13})|=
    |R_{13}-S_{13}|-|\overline{O}_{\Gamma}(T_{13})|+
|O_{\Gamma}(T_{13})|.
\end{equation*}
Then, since $|R_{13}-S_{13}|$ is even, we have that  $|O_{\Gamma}(\Theta_{13})|$ is an even integer. Now we have proved that $(\Gamma_{12}\cup\Gamma_{13})\cup(\Gamma_{2i}\cup\Gamma_{2j})$ is a spanning even subgraph of $\Gamma$ and every connected component of
$(\Gamma_{12}\cup\Gamma_{13})\cup(\Gamma_{2i}\cup\Gamma_{2j})$ has even number
of odd vertices of $\Gamma$ for some $i,j\in\{1,2,3\}$.
By Lemma \ref{evenly}, $\Gamma$ admits a nowhere-zero $4$-flow.
\end{proof}
As an application of Theorem \ref{union}, we obtain the following theorem.
\begin{theorem}
\label{4f}
Let $\Gamma=(V,E)$ be a graph. If every edge of $\Gamma$ is contained in a cycle of length at most $4$, then $\Gamma$ admits a nowhere-zero 4-flow.
\end{theorem}
\begin{proof}
We proceed the proof by induction on the size of $\Gamma$. If $\Gamma$ is a regular graph of degree $2$, then $\Gamma$ clearly admits a nowhere-zero 4-flow. In particular, $\Gamma$ admits a nowhere-zero 4-flow if it is of size $2$. Now we assume that $\Gamma$ is of size greater than $2$ and not a regular graph of degree $2$. Suppose that the theorem is true for graphs of size less than that of $\Gamma$. Since every edge of $\Gamma$ is contained in a cycle of length at most $4$, $\Gamma$ is a union of a number of cycles of length at most $4$. Let $\mathcal{F}$ be a set consisting of cycles of length at most $4$ of $\Gamma$ with minimal cardinality such that $\Gamma=\cup_{\Theta\in\mathcal{F}}\Theta$. Take a fixed $\Theta_{0}\in \mathcal{F}$ and set $\Sigma=\cup_{\Theta\in\mathcal{F}-\{\Theta_{0}\}}\Theta$. Then $\Theta_{0}$ admits a nowhere-zero 4-flow, $\Gamma=\Sigma\cup \Theta_{0}$ and every edge of $\Sigma$ is contained in a cycle in $\Sigma$ of length at most $4$. By the minimality of $\mathcal{F}$, $\Sigma$ is of size less than that of $\Gamma$. Therefore, by induction hypothesis, $\Sigma$ admits a nowhere-zero 4-flow.
Since $\Theta_{0}$ is a cycle of length at most $4$ and not a subgraph of $\Sigma$, $\Theta_{0}$ and $\Sigma$ have at most $3$ common edges. If $\Theta_{0}$ and $\Sigma$ have at most $2$ common edges, then by Corollary \ref{kcommon} $\Gamma$ admits a nowhere-zero $4$-flow.
Now we assume that $\Theta_{0}$ and $\Sigma$ have $3$ common edges. Then $\Theta_{0}$ is a $4$-cycle, and the set of common edges of $\Theta_{0}$ and $\Sigma$ induces a $3$-path in $\Gamma$. By Theorem \ref{union}, $\Gamma$ admits a nowhere-zero $4$-flow.
\end{proof}
It is worth mentioning that the following Catlin's lemma can be deduced from Corollary \ref{kcommon} and Theorem \ref{union} directly.
\begin{lem}
{\rm\cite[Lemma 3.1.18]{Z1997}}
A graph admits a nowhere-zero $4$-flow if it is a union of a cycle of length at most $4$ and a subgraph admitting a nowhere-zero $4$-flow.
\end{lem}
\begin{proof}
Suppose that $\Gamma$ is a union of two subgraphs $\Gamma_{1}$ and $\Gamma_{2}$ where $\Gamma_{1}$ is a cycle of length at most $4$ and $\Gamma_{2}$ admits a nowhere-zero $4$-flow. The lemma is clear true if $\Gamma=\Gamma_{2}$. Now we assume that $\Gamma_{2}$ is a proper subgraph of  $\Gamma$. Since $\Gamma_{1}$ is a cycle of length at most $4$, $\Gamma_{1}$ admits a nowhere-zero $4$-flow and $|E(\Gamma_{1})\cap E(\Gamma_{2})|\leq3$. If $|E(\Gamma_{1})\cap E(\Gamma_{2})|\leq2$, then by Corrolary \ref{kcommon} $\Gamma$ admits a nowhere-zero $4$-flow. If $|E(\Gamma_{1})\cap E(\Gamma_{2})|=3$, then  $E(\Gamma_{1})\cap E(\Gamma_{2})$ induces a $3$-path which is also a connected subgraph of $\Gamma$. By Theoreme \ref{union}, $\Gamma$ admits a nowhere-zero $4$-flow.
\end{proof}
At the end of the paper, we reprove the main result in \cite{IS2003} that every Cartesian bundle of two graphs without isolated vertices has a nowhere-zero $4$-flow by using Theorem \ref{4f}. As a generalisation of the concept of Cartesian product of graphs, the concept of Cartesian graph bundle was introduced in 1982 by Pisanski and Vrabec \cite{PV1982} and subsequently studied by
several authors; see for example \cite{IPZ1997,IS2003,KM1995,KL1990,MPS1988}.
The definition of Cartesian graph bundle is as follows.

\begin{defi}
\label{CB}
Let $\Sigma$ and $\Theta$ be two graphs. A graph $\Gamma$ is a Cartesian graph bundle with fibre $\Theta$ over the base graph $\Sigma$ if there is a mapping
$p$ from $V(\Gamma)\cup E(\Gamma)$ to $V(\Sigma)\cup E(\Sigma)$
satisfying the following conditions:
\begin{enumerate}
  \item It maps adjacent vertices of $\Gamma$ to adjacent or identical vertices in $V(\Sigma)$;
  \item The edges are mapped to edges or collapsed to a vertex;
  \item For every vertex $v\in V(\Sigma)$, $p^{-1}(v)$ is a subgraph of $\Gamma$ which is isomorphic to $\Theta$;
  \item For every edge $e\in E(\Sigma)$ with two ends $u$ and $v$, $p^{-1}(e)\cup p^{-1}(u)\cup p^{-1}(v)$ is a subgraph of $\Gamma$ which is isomorphic to $\mathcal{K}_2\square\Theta$ where $\mathcal{K}_2$ is the complete graph of order $2$.
\end{enumerate}
\end{defi}
\begin{theorem}[Rollov\'a and \v Skoviera]
\label{bundle}
Every Cartesian bundle of two graphs without isolated vertices has a nowhere-zero $4$-flow.
\end{theorem}
\begin{proof}
Let $\Sigma$ and $\Theta$ be two graphs without isolated vertices, and $\Gamma$ a Cartesian graph bundle with fibre $\Theta$ over the base graph $\Sigma$ associated with the mapping $p$. Let $\tilde{e}$ be an arbitrary edge of $\Gamma$. Then either $p(\tilde{e})$ is an edge in $E(\Sigma)$ or a vertex in $V(\Sigma)$. If $p(\tilde{e})$ is an edge $e\in E(\Sigma)$ with two ends $u$ and $v$, then $\tilde{e}$ is contained in $p^{-1}(e)\cup p^{-1}(u)\cup p^{-1}(v)$.  Since $p^{-1}(e)\cup p^{-1}(u)\cup p^{-1}(v)$ is a subgraph of $\Gamma$ isomorphic to $\mathcal{K}_2\square\Theta$ and $\Theta$ has no isolated vertices, every edge of $p^{-1}(e)\cup p^{-1}(u)\cup p^{-1}(v)$ is contained in a $4$-cycle and in particular $\tilde{e}$ is contained in a $4$-cycle. In what follows we assume $p(\tilde{e})$ is a vertex $u\in V(\Sigma)$. Since $\Sigma$ has no isolated vertices, there is a vertex $v\in V(\Sigma)$ and an edge $e\in E(\Sigma)$ linking $u$ and $v$. Therefore $\tilde{e}$ is contained in $p^{-1}(e)\cup p^{-1}(u)\cup p^{-1}(v)$ and it follows that $\tilde{e}$ is contained in a $4$-cycle. Now we have proved that every edge of of $\Gamma$ is contained in a $4$-cycle. By Theorem \ref{4f}, $\Gamma$ admits a nowhere-zero $4$-flow.
\end{proof}

\noindent {\textbf{Acknowledgements}}~~The first author was supported by the Basic Research and Frontier Exploration Project of Chongqing (No.~cstc2018jcyjAX0010) and the Foundation of Chongqing Normal University (21XLB006).
{\small

\end{document}